\documentclass[11pt]{article}
\usepackage[T1]{fontenc}
\usepackage{lmodern}
\usepackage{amsmath,amssymb,amsthm,mathtools}
\usepackage[margin=1in]{geometry}
\usepackage{booktabs,array,algorithm,algpseudocode}
\usepackage{graphicx,float}
\usepackage{placeins}
\usepackage[numbers,sort&compress]{natbib}
\usepackage{microtype,hyperref}
\hypersetup{colorlinks=true,linkcolor=blue,citecolor=blue,urlcolor=blue,
 pdftitle={Accelerating Finite-Sum Bilevel Optimization with Fully First-Order Oracles}}
\newtheorem{theorem}{Theorem}[section]
\newtheorem{lemma}[theorem]{Lemma}
\newtheorem{proposition}[theorem]{Proposition}

\newtheorem{assumption}{Assumption}
\theoremstyle{remark}

\newcommand{\E}{\mathbb E}
\newcommand{\R}{\mathbb R}
\newcommand{\norm}[1]{\lVert #1\rVert}
\newcommand{\ip}[2]{\langle #1,#2\rangle}
\newcommand{\Ot}{\widetilde{\mathcal O}}

\newcommand{\Lip}{\operatorname{Lip}}
\DeclareMathOperator*{\argmin}{arg\,min}
\algrenewcommand\algorithmicrequire{\textbf{Input:}}
\algrenewcommand\algorithmicensure{\textbf{Output:}}
\title{Accelerating Finite-Sum Bilevel Optimization\\with Fully First-Order Oracles}
\author{
  Qihao Zhou\thanks{Independent Researcher. Corresponding author. Email: \texttt{qihao\_zhou@berkeley.edu}.}
  \and
  Shanshui Niu\thanks{Independent Researcher. Email: \texttt{shanshun@alumni.cmu.edu}.}
}
\date{}
\newcommand{\draftfigure}[1]{%
 \IfFileExists{#1}{\includegraphics[width=\textwidth]{#1}}{%
  \fbox{\parbox[c][32mm][c]{0.95\linewidth}{\centering
  Figure file not supplied\\[4pt]\texttt{\detokenize{#1}}}}}%
}
\begin{document}
\maketitle
\vspace{-1.5em}
\begin{abstract}
We study nonconvex--strongly-convex bilevel optimization with finite-sum objectives and fully first-order oracles. We propose Sampled Implicit Corrections for Bilevel Optimization (SICBO), which combines a penalty approximation with variance-reduced corrections for the lower-level variables. Under componentwise lower strong convexity, SICBO finds a point $x$ satisfying $\E\norm{\nabla\Phi(x)}^2\le\varepsilon^2$ within $\Ot(N+\sqrt N\kappa^{7/2}\varepsilon^{-2})$ component-gradient calls, where $N$ is the total number of components and $\kappa$ is the normalized lower-level condition number. The bound achieves square-root finite-sum dependence using only gradients and preserves the best known deterministic fully first-order condition-number exponent for smooth problems. Our improvement follows from a curvature-dependent bound for nonlinear correction errors, which permits epochs of length $\widetilde\Theta(\sqrt{N/\kappa})$ with batches of size $\widetilde\Theta(\sqrt{N\kappa})$. A single-level reduction shows that the dependence on $N$ and $\varepsilon$ is optimal up to logarithmic factors for fixed $\kappa$.
\end{abstract}

\section{Introduction}\label{sec:intro}
We study the finite-sum bilevel optimization problem
\begin{equation}\label{eq:problem}
 \min_{x\in\R^{d_x}}\Phi(x):=f(x,y^*(x)),\qquad
 y^*(x):=\argmin_{y\in\R^{d_y}}g(x,y),
\end{equation}
where $f(x,y)=n^{-1}\sum_{i=1}^n f_i(x,y)$ and
$g(x,y)=m^{-1}\sum_{j=1}^m g_j(x,y)$. The lower objective is strongly convex in $y$, and $\Phi$ may be nonconvex. Such problems arise when an upper variable is selected through the solution of a training problem. We write $N=n+m$ and count individual component-gradient evaluations.

The gradient of $\Phi$, called the hypergradient, involves the inverse of the lower Hessian. Implicit differentiation methods therefore typically require Hessian-vector products. Fully first-order methods approximate the hypergradient through differences of gradients. The penalty approach of \citet{kwon} and \citet{cmz}, combined with accelerated inner minimization, gives a deterministic complexity of $\Ot(\kappa^{7/2}\varepsilon^{-2})$ \citep{cjz}.

Implementing these deterministic methods on \eqref{eq:problem} by full gradients adds a factor $N$. Accelerated finite-sum solvers reduce the inner minimization cost, but a full pass at every outer iteration still contributes $N\kappa^3\varepsilon^{-2}$. This term exceeds the desired finite-sum rate when $N>\kappa$.

In this paper, we propose Sampled Implicit Corrections for Bilevel Optimization (SICBO), which tracks the two inner solutions with sampled corrections. The key step is to control sampling errors through these nonlinear corrections in curvature-dependent norms. The resulting epoch estimate gives the complexity
\begin{equation}\label{eq:headline}
 \Ot\!\left(N+\sqrt N\kappa^{7/2}\varepsilon^{-2}\right).
\end{equation}
The bound combines gradient-only access, square-root finite-sum dependence, and the $\kappa^{7/2}$ exponent of the accelerated deterministic penalty method. The analysis uses curvature-weighted concentration at population endpoints and predictable energy peeling to control the error over each epoch. We also prove a lower bound matching the dependence on $N$ and $\varepsilon$ up to logarithms for fixed $\kappa$; a matching bound in $\kappa$ remains open. On a synthetic finite-sum family, SICBO uses the fewest component-gradient calls among the methods in Table~\ref{tab:complexity}, with an advantage over the full-refresh benchmark that decreases with $\kappa$ as predicted by the amortized cost (Section~\ref{sec:numerics}).

\section{Preliminaries}\label{sec:prelim}
We use $\norm\cdot$ for the Euclidean norm of vectors and the induced operator norm of matrices. For a positive definite matrix $H$, let $\norm u_H^2=u^\top Hu$. We write $a\lesssim b$ to suppress an absolute constant and $\Ot$ to suppress logarithmic factors. All polynomial dependence on $N$, $\kappa$, and $\varepsilon$ is explicit. We impose the following assumptions on the whole space.

\begin{assumption}[Component smoothness]\label{ass:smooth}
Each $f_i$ has a $1$-Lipschitz continuous gradient. Each $g_j$ is twice continuously differentiable, its gradient is $1$-Lipschitz continuous, and its Hessian is $1$-Lipschitz continuous. These Lipschitz conditions apply jointly to $(x,y)$.
\end{assumption}
\begin{assumption}[Lower strong convexity]\label{ass:sc}
For every $x$ and $j$, the function $g_j(x,\cdot)$ is $\mu$-strongly convex, where $0<\mu\le1$. We define $\kappa=1/\mu$.
\end{assumption}
\begin{assumption}[Upper gradient and initial gap]\label{ass:gap}
The average objective satisfies $\norm{\nabla_y f(x,y)}\le1$ for all $(x,y)$. Furthermore, $\Phi$ is bounded below and the initial point satisfies $\Phi(x_0)-\inf_x\Phi(x)\le\Delta$, where a known upper bound $\Delta=O(1)$ is available.
\end{assumption}
The unit bounds in Assumptions~\ref{ass:smooth} and \ref{ass:gap} define the normalized class for our condition-number dependence; no powers of other smoothness parameters are suppressed. Componentwise strong convexity makes every sampled correction strongly convex and permits accelerated finite-sum solves. Strong convexity of the average $g$ alone does not ensure this property. We require neither bounded component-gradient variance nor Hessian smoothness of $f_i$.

Our goal is to produce a random point $x_{\rm out}$ such that
\begin{equation}\label{eq:criterion}
 \E\norm{\nabla\Phi(x_{\rm out})}^2\le\varepsilon^2,
 \qquad 0<\varepsilon\le1.
\end{equation}
One oracle call returns $\nabla f_i(x,y)$ or $\nabla g_j(x,y)$ at a specified index and point, so a full gradient costs $O(N)$ calls. The oracles are exact; arithmetic precision is excluded from the complexity.

For $\lambda\ge2\kappa$, define the penalty value function
\begin{equation}\label{eq:penalty}
 H_\lambda(x):=\min_y\{f(x,y)+\lambda g(x,y)\}
                 -\lambda\min_z g(x,z).
\end{equation}
The following lemma gives the approximation error and smoothness of $H_\lambda$. Its proof appears in Appendix~\ref{app:fields}.
\begin{lemma}[Penalty approximation]\label{lem:penalty}
Under Assumptions~\ref{ass:smooth}--\ref{ass:gap},
\begin{equation}\label{eq:penalty-bounds}
\begin{aligned}
 \norm{\nabla H_\lambda(x)-\nabla\Phi(x)}&\le C\kappa^3/\lambda,
 \qquad \Lip(\nabla H_\lambda)\le C_H\kappa^3,\\
 \Phi(x)-\kappa/(2\lambda)&\le H_\lambda(x)\le\Phi(x).
\end{aligned}
\end{equation}
In particular, the initial gap of $H_\lambda$ is at most
$\bar\Delta_H:=\Delta+\kappa/(2\lambda)=O(1)$.
\end{lemma}

\section{Related Work}\label{sec:related}
\paragraph{Applications.} Bilevel problems arise in hyperparameter optimization \citep{pedregosa,franceschi}, training-example reweighting \citep{ren2018}, and meta-learning \citep{franceschi}. In the validation-based setting of Section~\ref{sec:num-real}, the upper and lower sums are validation and training losses, respectively.

\paragraph{Hypergradient methods.} Implicit hypergradient methods use Hessian- or Jacobian-vector products of the lower objective. Stochastic approximation methods were analyzed by \citet{ghadimi}, \citet{hong2023} and \citet{ji2021}; momentum and recursive variance reduction improve their sample complexity \citep{khanduri2021,yang2021}. These methods track both the lower solution and a linear-system solution.

\paragraph{Fully first-order methods.} Fully first-order methods replace second-order products by differences of gradients. \citet{liu2022} proposed a value-function method, and \citet{kwon} developed a penalty-based stochastic approximation method. \citet{kwon2024} studied the penalty landscape, and \citet{shen2023} analyzed penalty gradient descent without lower-level strong convexity. Under normalized smoothness assumptions, \citet{cmz} obtained a deterministic gradient complexity of $\Ot(\kappa^4\varepsilon^{-2})$. Accelerated inner solves improve this bound to $\Ot(\kappa^{7/2}\varepsilon^{-2})$ \citep[Appendix F]{cjz}. Full-pass implementations give the first two rows of Table~\ref{tab:complexity}.

\begin{table}[H]
\centering
\caption{Component-gradient complexity for achieving \eqref{eq:criterion}. All expressions display the dependence on $N$, $\kappa$, and $\varepsilon$; logarithmic factors are suppressed. The first two rows implement deterministic methods by full passes. The last two rows assume componentwise lower strong convexity. Every row uses component gradients only.}\label{tab:complexity}
\small
\renewcommand{\arraystretch}{1.45}
\begin{tabular}{@{}>{\raggedright\arraybackslash}p{0.44\linewidth}@{\hspace{0.06\linewidth}}>{\centering\arraybackslash}p{0.50\linewidth}@{}}
\toprule
Method & Component-gradient complexity\\
\midrule
Penalty method of \citet{cmz} & $\Ot\!\left(N\kappa^4\varepsilon^{-2}\right)$\\
Accelerated penalty method \citep{cjz} & $\Ot\!\left(N\kappa^{7/2}\varepsilon^{-2}\right)$\\
\textbf{Full refresh with accelerated finite-sum solves (Proposition~\ref{prop:baseline})} &
{\bfseries\boldmath $\Ot\!\left(N+\bigl(N\kappa^3+\sqrt N\kappa^{7/2}\bigr)\varepsilon^{-2}\right)$}\\
\textbf{SICBO (Theorem~\ref{thm:main})} &
{\bfseries\boldmath $\Ot\!\left(N+\sqrt N\kappa^{7/2}\varepsilon^{-2}\right)$}\\
\bottomrule
\end{tabular}
\end{table}

\paragraph{Finite-sum bilevel methods.} SVRG \citep{svrg}, SARAH \citep{sarah}, SPIDER \citep{spider} and PAGE \citep{page} reduce full-gradient evaluations in finite-sum optimization. SABA \citep{saba} and SRBA \citep{srba} apply variance reduction to the implicit hypergradient. SPABA \citep{spaba} achieves $O(N+\sqrt N\varepsilon^{-2})$ under \eqref{eq:criterion}, including initialization and suppressing other problem constants, with gradient and Hessian/Jacobian-vector oracles. Its $\kappa$ dependence is not explicit. Our bound uses only component gradients and gives explicit $\kappa^{7/2}$ dependence, under the stronger assumptions of componentwise lower strong convexity and globally bounded average $\nabla_y f$.

\paragraph{Accelerated finite-sum minimization.} Katyusha \citep{katyusha} solves a smooth, strongly convex finite sum with $M$ components and condition number $O(\kappa)$ in $\Ot(M+\sqrt{M\kappa})$ gradient calls. This result immediately yields the following benchmark.
\begin{proposition}[Full-refresh benchmark]\label{prop:baseline}
Solving both inner problems and computing a full penalty-gradient estimate at every outer iteration gives the component-gradient complexity
\begin{equation}\label{eq:baseline}
 \Ot\!\left(N+\bigl(N\kappa^3+\sqrt N\kappa^{7/2}\bigr)
                    \varepsilon^{-2}\right).
\end{equation}
For $N\le\kappa$, this is bounded by \eqref{eq:headline}.
\end{proposition}
The proof appears in Appendix~\ref{app:cost}. For $N>\kappa$, SICBO improves the iterative cost $\Ot(N\kappa^3\varepsilon^{-2})$ by a factor $\widetilde\Theta(\sqrt{N/\kappa})$. The total-cost improvement can be smaller when initialization dominates.

\paragraph{Lower bounds.} \citet{jiliang2023} established lower bounds for bilevel problems with strongly convex or convex upper objectives. For our nonconvex finite-sum class, the following reduction gives the optimal dependence on $N$ and $\varepsilon$.
\begin{proposition}[Optimal dependence for fixed condition number]\label{prop:lower}
For every fixed $\kappa\ge1$, sufficiently large $N$, and sufficiently small $\varepsilon>0$, the worst-case component-gradient complexity of \eqref{eq:criterion} under Assumptions~\ref{ass:smooth}--\ref{ass:gap}, with $\Delta=1$, is at least $\Omega(N+\sqrt N\varepsilon^{-2})$.
\end{proposition}
Appendix~\ref{app:lower} proves the proposition by reduction from the individual-smoothness lower bound of \citet{peng}, which permits unrestricted randomized component queries. Thus \eqref{eq:headline} is optimal up to logarithms in $N$ and $\varepsilon$ for fixed $\kappa$. \citet[Theorem~2]{srba} previously proved an $\Omega(\sqrt n\varepsilon^{-2})$ bound for a restricted class of linear bilevel algorithms, where $n$ is our upper-level sum size. Their regime is $\varepsilon^2\lesssim1/n$ with other constants fixed, and their result does not characterize $\kappa$ dependence.

For condition-number dependence, \citet[Theorem~4.1]{cjz} prove an $\Omega(\kappa^{5/2}\varepsilon^{-2})$ bound for deterministic first-order algorithms with quadratic lower objectives. Their class does not impose our global bound on $\nabla_y f$, so the result does not directly apply here. Appendix~\ref{app:lower} gives further lower bounds under our assumptions. The sharp joint dependence on $N$ and $\kappa$ remains open; separate lower bounds do not establish a product bound.

\section{The SICBO Method}\label{sec:method}
\subsection{Scaled gradient differences}
Let $z$ approximate $y^*(x)$ and introduce $v=\lambda(y-z)$. For $w=(x,z,v)$, define
\begin{equation}\label{eq:fields}
\begin{aligned}
 A_j(x,z)&=\nabla_y g_j(x,z),\\
 B_{ij}(w)&=\nabla_y f_i(x,z+v/\lambda)
       +\lambda\bigl[\nabla_y g_j(x,z+v/\lambda)-\nabla_y g_j(x,z)\bigr],\\
 C_{ij}(w)&=\nabla_x f_i(x,z+v/\lambda)
       +\lambda\bigl[\nabla_x g_j(x,z+v/\lambda)-\nabla_x g_j(x,z)\bigr].
\end{aligned}
\end{equation}
Write $A,B,C$ for the population means and $A_S,B_S,C_S$ for sample averages, with $A(w)=A(x,z)$. Each batch samples independent uniform pairs $(i,j)$ with replacement, and $A_S$ ignores $i$. Each lower-gradient difference uses the same index $j$. Computing full means costs $O(N)$ calls.

Let $z^*(x)=y^*(x)$ and let $v^*(x,z)$ be the unique root of $B(x,z,\cdot)$. At these roots,
\begin{equation}\label{eq:exact-field}
 C\bigl(x,z^*(x),v^*(x,z^*(x))\bigr)=\nabla H_\lambda(x).
\end{equation}
The scaling makes the inner condition number independent of $\lambda$: each $B_{ij}(x,z,\cdot)$ is the gradient of a strongly convex potential with condition number $O(\kappa)$. The following lemma gives the field regularity.
\begin{lemma}[Field regularity]\label{lem:fields}
For $\lambda\ge2\kappa$, $B_{ij}(x,z,\cdot)$ is $(\mu/2)$-strongly monotone and $2$-Lipschitz. Moreover,
\begin{equation}\label{eq:root-lipschitz}
 \norm{v^*(x,z)}\le\kappa,\qquad
 \Lip(z^*)\le\kappa,\qquad \Lip_z(v^*)\lesssim\kappa^2.
\end{equation}
On $\norm v\le2\kappa$, each $B_{ij}$ and $C_{ij}$ is $O(\kappa)$-Lipschitz in $(x,z)$ and $O(1)$-Lipschitz in $v$. Consequently,
$z\mapsto C(x,z,v^*(x,z))$ is $O(\kappa^2)$-Lipschitz.
\end{lemma}
The proof appears in Appendix~\ref{app:fields}. The fields in \eqref{eq:fields} require only component gradients.

\subsection{Sampled implicit corrections}
SICBO maintains $z_t,v_t$ and the field estimates $\widehat A_t,\widehat B_t,\widehat G_t$. Write $w_t=(x_t,z_t,v_t)$. Each epoch begins with population inner updates and full field evaluations. Between refreshes, SICBO updates $x$, then $z$, then $v$, using the residual maps
\begin{equation}\label{eq:correction-residuals}
\begin{aligned}
 r_t^z(u)&:=A_{S_t^z}(x_{t+1},u)-A_{S_t^z}(x_t,z_t)+\widehat A_t,\\
 r_t^v(u)&:=B_{S_t^v}(x_{t+1},z_{t+1},u)-B_{S_t^v}(w_t)+\widehat B_t.
\end{aligned}
\end{equation}
We use capped Katyusha iterations to obtain $\norm{r_t^z(z_{t+1})}\le\tau$ and then, with $z_{t+1}$ fixed, $\norm{r_t^v(v_{t+1})}\le\tau$. Each residual is the gradient of a strongly convex objective. Every batch contains $b$ fresh, independent uniform index pairs. The evaluation batch $R_t$ is drawn after both corrections, making its increments conditionally centered at the updated points.

\begin{algorithm}[t]
\caption{SICBO: Sampled Implicit Corrections for Bilevel Optimization}\label{alg:main}
\begin{algorithmic}[1]
\Require $x_0,z_{\rm init},v_{\rm init};\quad T,q,b,\lambda,\alpha,\Gamma,\tau$.
\For{$t=0,\ldots,T-1$}
 \If{$t\equiv0\pmod q$}
  \State Run Katyusha in $z_t$ until $\norm{A(x_t,z_t)}\le\tau$.
  \State Run Katyusha in $v_t$ until $\norm{B(x_t,z_t,v_t)}\le\tau$.
  \State $(\widehat A_t,\widehat B_t,\widehat G_t)\gets\bigl(A(x_t,z_t),B(w_t),C(w_t)\bigr)$.
 \EndIf
 \State $g_t\gets\widehat G_t/\max\{1,\norm{\widehat G_t}/\Gamma\}$.
 \State $x_{t+1}\gets x_t-\alpha g_t$.
 \If{$t+1<T$ and $t+1\not\equiv0\pmod q$}
  \State Draw $S_t^z$; run Katyusha in $z_{t+1}$ until $\norm{r_t^z(z_{t+1})}\le\tau$.
  \State Draw $S_t^v$; run Katyusha in $v_{t+1}$ until $\norm{r_t^v(v_{t+1})}\le\tau$.
  \State Draw $R_t$; set $(\widehat A,\widehat B,\widehat G)_{t+1}\gets(\widehat A,\widehat B,\widehat G)_t+(A,B,C)_{R_t}(w_{t+1})-(A,B,C)_{R_t}(w_t)$.
 \EndIf
\EndFor
\State $\mathcal U\gets\{t<T:\norm{\widehat G_t}\le\Gamma\}$.
\State \Return $x_\xi$ with $\xi\sim\operatorname{Unif}(\mathcal U)$, or $x_0$ if $\mathcal U=\varnothing$.
\end{algorithmic}
\end{algorithm}

The first inner solves start at $z_{\rm init},v_{\rm init}$ and later solves at the preceding iterates. We cap every solve and check its residual. Sampled updates must also satisfy $\norm{z_{t+1}-z_t}\le1/(512\kappa)$ and $\norm{v_{t+1}}\le2\kappa$. A failed test returns $x_0$. The output set $\mathcal U$ uses the estimates that generated the outer steps.

Clipping and the inner checks bound all accepted trajectories, allowing the high-probability tracking estimates to imply \eqref{eq:criterion}. A failed check terminates the run without an additional full pass. Appendix~\ref{app:cost} gives the solver budgets and their cost. For $N\le\kappa$, we set $q=1$ and recover the full-refresh benchmark.

\section{Complexity Analysis}\label{sec:analysis}
\subsection{Parameters and main result}
For $N>\kappa$, choose
\begin{equation}\label{eq:parameters}
\begin{alignedat}{3}
 q&=\left\lceil\sqrt{N/\kappa}\right\rceil,\qquad&
 b&=\left\lceil C_b\kappa q\Lambda^2\right\rceil,\qquad&
 T&=\left\lceil C_T\kappa^3\varepsilon^{-2}\right\rceil,\\
 \lambda&=C_\lambda\kappa^3/\varepsilon,&
 \alpha&=c_\alpha\kappa^{-3},&
 \tau&=c_\tau\varepsilon\kappa^{-4},\\
 \theta&=c_\theta\kappa^{-2},&
 \Gamma&=\theta/\alpha,&
 \nu&=c_\nu\varepsilon\kappa^{-3}.
\end{alignedat}
\end{equation}
Here $\nu$ is used only in the analysis. To specify the confidence level, put $L_H=C_H\kappa^3$, $d=d_x+d_y+2$, and
\begin{equation}\label{eq:confidence}
 B_T^2=4L_H\bar\Delta_H+2L_H^2T^2\theta^2,\quad
 \delta=\min\left\{\frac1{10},\frac{\varepsilon^2}{64(1+B_T^2)}\right\},\quad
 \Lambda=\log\frac{e d N\kappa(T+1)}{\varepsilon\delta}.
\end{equation}
All constants are absolute. We choose $c_\theta$ small and $C_\lambda$ large, then $c_\alpha\le\min\{c_\theta/10,1/(4C_H)\}$, then $c_\tau,c_\nu$ small and $C_b$ large, and finally $C_T$ large. In particular, $\Gamma\ge10\kappa$ and every outer displacement is at most $\theta$. Each inner solve receives failure probability at most $\delta/(C(T+1))$ for a sufficiently large absolute $C$.

\begin{theorem}[Component-gradient complexity]\label{thm:main}
Under Assumptions~\ref{ass:smooth}--\ref{ass:gap}, SICBO (Algorithm~\ref{alg:main}), with \eqref{eq:parameters}--\eqref{eq:confidence} for $N>\kappa$ and $q=1$ for $N\le\kappa$, returns a point satisfying
\begin{equation}\label{eq:main-guarantee}
 \E\norm{\nabla\Phi(x_{\rm out})}^2\le\varepsilon^2
\end{equation}
within
\begin{equation}\label{eq:main-complexity}
 \Ot\!\left(N+\sqrt N\kappa^{7/2}\varepsilon^{-2}\right)
\end{equation}
component-gradient calls. For $N\le\kappa$, the remaining parameters are unchanged and the unused batch size can be omitted. The hidden logarithms include dimension, solver confidence, and the initial inner-distance bound. No Hessian or Hessian-vector oracle is used.
\end{theorem}
The proof appears in Appendices~\ref{app:local}--\ref{app:cost}. The result allows nonlinear lower objectives and noncommuting component Hessians. Componentwise strong convexity is used in both the sampled root analysis and the accelerated finite-sum solves.

The epoch estimate permits $b=\widetilde\Theta(\kappa q)$, giving the amortized cost
\begin{equation}\label{eq:balance}
 \Ot(N/q+b)=\Ot(N/q+\kappa q).
\end{equation}
Balancing the terms gives $q\asymp\sqrt{N/\kappa}$ and cost $\Ot(\sqrt{N\kappa})$ per step. Together with $T=O(\kappa^3\varepsilon^{-2})$, this proves \eqref{eq:main-complexity}. An additional polynomial factor of $\kappa$ in the batch size would increase the final condition-number exponent.

\subsection{An epoch estimate for the nonlinear corrections}
We outline the epoch estimate that determines the condition-number dependence. Index a correction by its endpoint $t$ and let
\[
 H^A_t=\nabla^2_{yy}g(x_t,z_{t-1}),\qquad
 H^B_t=\nabla^2_{yy}g(x_t,z_t+v_{t-1}/\lambda).
\]
These matrices are used only in the analysis. For a tracker $F$, let $\bar u_t$ be the population corrected root, $u_t$ its sampled approximation, and $d_t=u_t-u_{t-1}$. Define
\begin{equation}\label{eq:energy-definitions}
 Q_t=\norm{\bar u_t-u_{t-1}}_{H_t}^2,\qquad
 E_t=\norm{u_t-\bar u_t}_{H_t}^2,\qquad
 J_t=\norm{d_t}_{H_t}^2.
\end{equation}
Let $D_t$ bound the component change caused by the new context at the old inner point. Conditional concentration and local curvature comparison give
\begin{equation}\label{eq:main-energy}
 E_t\lesssim\frac{\kappa\Lambda}{b}(Q_t+D_t^2)+\kappa^2\tau^2.
\end{equation}
The population endpoint in \eqref{eq:main-energy} is fixed before its correction batch is drawn, so the concentration bound is applied at a sample-independent point.

The independent evaluation batch gives a martingale error $M_t=F(w_t)-\widehat F_t$. Predictable stopping and peeling bound $\norm{M_t}^2$ by the realized energy $\sum_{s\le t}(J_s+D_s^2)$. Combining this bound with \eqref{eq:main-energy} gives the following lemma.
\begin{lemma}[Epoch error]\label{lem:epoch-main}
There is an event $\mathcal T$ with probability at least $1-\delta$ on which all safeguards pass and, for every epoch $\mathcal E$ of actual length $h\le q$,
\begin{equation}\label{eq:epoch-main}
 \sum_{t\in\mathcal E}\norm{\widehat G_t-\nabla H_\lambda(x_t)}^2
 \le C\frac{\kappa^7q\Lambda}{b}
       \sum_{t\in\mathcal E}\norm{x_{t+1}-x_t}^2
       +Ch\bigl(\kappa^7\tau^2+\kappa^6\nu^2\bigr).
\end{equation}
The inequality holds pathwise on $\mathcal T$.
\end{lemma}
The proof appears in Appendices~\ref{app:local}--\ref{app:epoch}. Let $X$ be the outer squared movement over an epoch and omit tolerance terms. The first tracker gives $\sum\norm{\Delta z_t}^2\lesssim\kappa^2X$. The second then receives context energy $O(\kappa^4X)$ and contributes error $O(\kappa^7q\Lambda X/b)$. The recursive error in $C$ is $O(\kappa^6q\Lambda X/b)$.

For the descent argument, smoothness gives
\begin{equation}\label{eq:main-descent}
 H_\lambda(x_{t+1})\le H_\lambda(x_t)
 -\frac\alpha4\norm{g_t}^2
 +\frac\alpha2\norm{\widehat G_t-\nabla H_\lambda(x_t)}^2.
\end{equation}
Substituting \eqref{eq:epoch-main} gives the movement coefficient
$C\alpha^2\kappa^7q\Lambda/b=Cc_\alpha^2\kappa q\Lambda/b$, which is small enough to absorb. The resulting bound ensures at least $T/2$ unclipped steps on $\mathcal T$, controlling the denominator in uniform sampling from $\mathcal U$. Clipping bounds the gradient at any returned point by $B_T$, so $\mathcal T^c$ contributes at most $\delta B_T^2$. Lemma~\ref{lem:penalty} transfers stationarity from $H_\lambda$ to $\Phi$.

\section{Numerical Experiments}\label{sec:numerics}
We compare practical implementations of the four methods in Table~\ref{tab:complexity}. All use $H_\lambda$ with $\lambda=\kappa^3/\varepsilon$ and outer updates $x_{t+1}=x_t-\alpha g_t$. The two penalty methods use full-pass gradient descent and accelerated gradient descent, respectively. The full-refresh method uses $q=1$ in Algorithm~\ref{alg:main}; SICBO uses $q=\lceil\sqrt{N/\kappa}\rceil$ and $b=\lceil\kappa q\rceil$. We set $\tau=\varepsilon\kappa^{-4}$ on the synthetic problem and use the tolerance floor below on real data. SICBO monitors the movement and radius checks without terminating at violations. Appendix~\ref{app:numerics} details these implementations and their departures from \eqref{eq:parameters}.

\begin{figure}[t]
\centering
\draftfigure{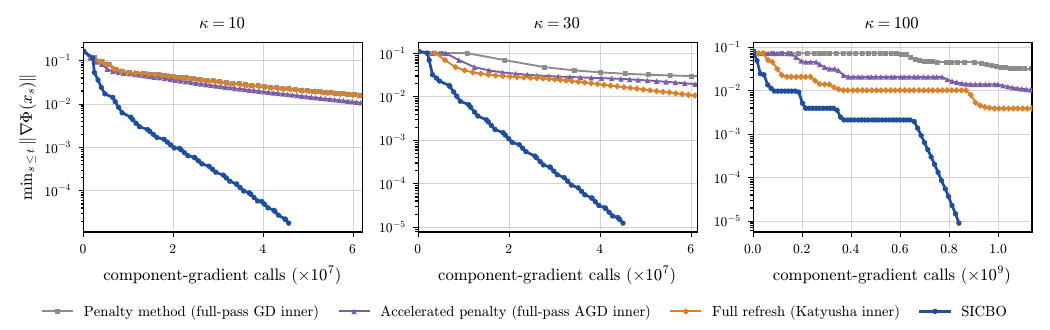}
\caption{Running minimum of $\norm{\nabla\Phi(x_t)}$ against component-gradient calls on \eqref{eq:comparison-family}. Markers are outer iterations. For the full-refresh benchmark and SICBO, each curve averages ten seeds; the two penalty methods are deterministic.}\label{fig:comparison}
\end{figure}

\subsection{Synthetic problem}\label{sec:num-synthetic}
We take $n=m=10^4$, so that $N=2\cdot10^4$, $d_x=d_y=8$, and $\kappa\in\{10,30,100\}$. The components are
\begin{equation}\label{eq:comparison-family}
\begin{aligned}
 g_j(x,y)&=\frac12y^\top H_jy-y^\top K_jx+b_j^\top y+\psi(w^\top y),
 &\psi''(u)&=\mu\sin\frac{u}{\sqrt\mu},\\
 f_i(x,y)&=\frac{7}{10}\sum_{k}\rho(y_k-t_{ik})+\frac{1}{50}\sum_k\omega(x_k),
 &\rho(t)&=\frac12\log(1+t^2),\quad \omega(t)=\frac{t^2}{1+t^2},
\end{aligned}
\end{equation}
where the component Hessians $H_j$ do not commute and $K_j=H_jM+E_j$ for a fixed matrix $M$ (Appendix~\ref{app:num-synthetic}). All methods use a common step $\alpha$, and the target is $\varepsilon=10^{-4}\norm{\nabla\Phi(x_0)}$.

Figure~\ref{fig:comparison} shows that all methods take essentially the same number of outer iterations. SICBO has the lowest cost per iteration and reaches the target with the fewest calls for every $\kappa$. Its advantage over the full-refresh benchmark decreases at the rate $\kappa^{-1/2}$ predicted by \eqref{eq:balance} (Appendix~\ref{app:num-synthetic}).

\subsection{Real data}\label{sec:num-real}
We also consider per-feature ridge-weight selection for logistic regression, a finite-sum benchmark used by \citet{saba,srba,spaba}. Let $(a_j,b_j)$ and $(a'_i,b'_i)$ be the $m$ training and $n$ validation rows, respectively, with $\norm{a_j}=\norm{a'_i}=1$ and labels in $\{\pm1\}$. Define $\phi(u)=\log(1+e^{-u})$ and
\begin{equation}\label{eq:real-reg}
 g_j(x,y)=\phi\bigl(b_j\ip{a_j}{y}\bigr)+\frac12\sum_{k=1}^{p}r_k(x)\,y_k^2,\qquad
 f_i(x,y)=\phi\bigl(b'_i\ip{a'_i}{y}\bigr),\qquad
 r_k(x)=\mu+\frac{9\mu}{1+e^{-x_k}},
\end{equation}
so $x\in\R^p$ sets one ridge weight per feature. We use the LIBSVM data sets a9a, w8a and ijcnn1 \citep{libsvm}, with $N$ from $3.3\cdot10^{4}$ to $1.4\cdot10^{5}$, $\kappa\in\{100,1000\}$ and $\varepsilon=10^{-3}\norm{\nabla\Phi(x_0)}$. The inner tolerance is $\tau=\max\{\varepsilon\kappa^{-4},10^{-11}\}$. Each method uses its own tuned step, and the penalty methods are stopped after three times SICBO's mean cost at the reference step multiplier $c=1$ (Appendix~\ref{app:num-real}). These real-data objectives do not satisfy Assumption~\ref{ass:smooth} globally, as detailed in the same appendix.

\begin{figure}[t]
\centering
\draftfigure{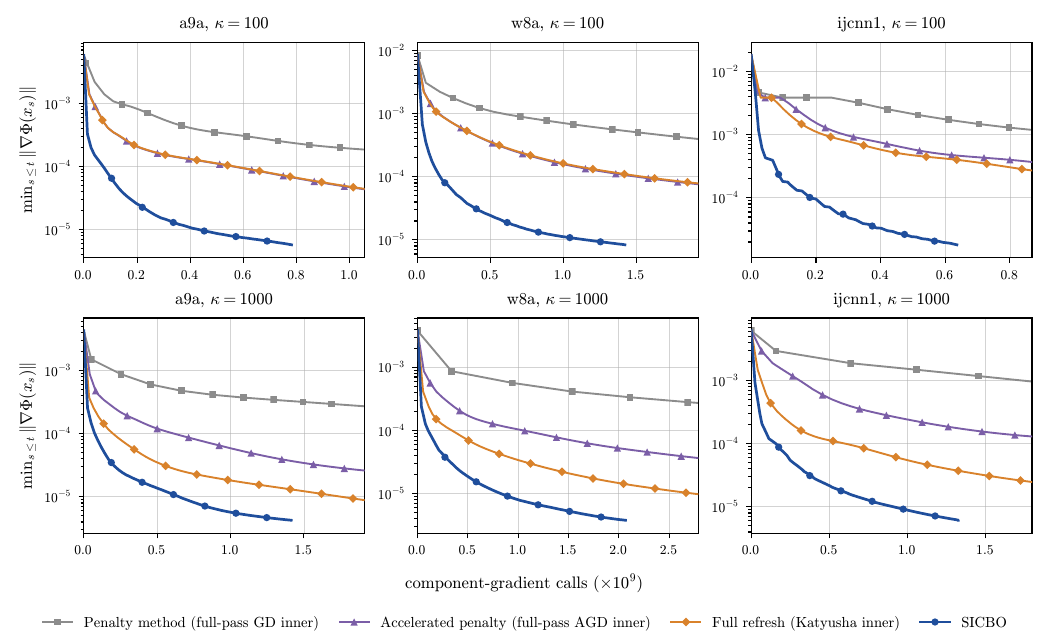}
\caption{Running minimum of $\norm{\nabla\Phi(x_t)}$ against component-gradient calls on \eqref{eq:real-reg}, with the step tuned for each method. Lines join the outer iterations; markers are drawn at a few of them and are offset between methods, so that coinciding curves remain distinguishable. For the full-refresh benchmark and SICBO, each curve averages ten seeds; the penalty methods are deterministic.}\label{fig:real}
\end{figure}

Figure~\ref{fig:real} shows that SICBO reaches the target with the fewest calls in all six settings. The full-refresh benchmark is second, except on a9a and w8a with $\kappa=100$, where the accelerated penalty method has similar cost. On each data set, SICBO's advantage over full refresh decreases from $\kappa=100$ to $\kappa=1000$ and remains nearly the same fraction of $\sqrt{N/\kappa}$, consistent with \eqref{eq:balance} (Appendix~\ref{app:num-real}).

\FloatBarrier
\begin{samepage}
\section{Conclusion}\label{sec:conclusion}
We proposed SICBO for finite-sum nonconvex--strongly-convex bilevel optimization. Its component-gradient complexity $\Ot(N+\sqrt N\kappa^{7/2}\varepsilon^{-2})$ is optimal up to logarithms in $N$ and $\varepsilon$ for fixed $\kappa$ and preserves the best known deterministic fully first-order condition-number exponent for smooth problems. The analysis combines curvature-dependent correction bounds with pathwise control of evaluation errors. Experiments on synthetic and LIBSVM data show the fewest component-gradient calls among the compared methods, with an advantage over full refresh that decreases as $\kappa^{-1/2}$. Extending the result to average lower strong convexity, obtaining a single-loop method, and establishing the sharp joint dependence on $N$ and $\kappa$ remain open.

\label{page:main-end}
\par
\end{samepage}
\clearpage
\section*{Acknowledgements}
The authors thank Lesi Chen of Tsinghua University for helpful discussions of this open problem. Lesi Chen showed the initial upper bound stated in Proposition~\ref{prop:baseline} and conjectured the sharper bound established in this manuscript.

Codex with GPT-6 Astra and Claude Science Opus 5.5 were used to assist with the development of the construction and proofs, numerical checks of algebraic identities, the review of related literature, and the preparation and revision of this manuscript. The authors are responsible for the mathematical claims, references, and final presentation.


\clearpage
\appendix
\numberwithin{equation}{section}
\section{Properties of the Component Fields and Penalty}\label{app:fields}
\subsection{Proof of Lemma~\ref{lem:fields}}
For fixed $(x,z)$, the field $B_{ij}(x,z,\cdot)$ is the gradient of
\begin{equation}\label{eq:component-potential}
 \varphi_{ij}(v)=\lambda f_i(x,z+v/\lambda)
 +\lambda^2[g_j(x,z+v/\lambda)-g_j(x,z)]
 -\lambda\ip{\nabla_y g_j(x,z)}v.
\end{equation}
The first term is $1/\lambda$-smooth and $1/\lambda$-weakly convex; the second is $\mu$-strongly convex and $1$-smooth. Thus $\varphi_{ij}$ is $(\mu-1/\lambda)$-strongly convex and $(1+1/\lambda)$-smooth, giving the claimed constants for $\lambda\ge2\kappa$.

Let $y=z+v^*/\lambda$. The population root equation gives
\[
 \lambda[\nabla_y g(x,y)-\nabla_y g(x,z)]=-\nabla_y f(x,y).
\]
Taking the inner product with $y-z$, strong convexity and Assumption~\ref{ass:gap} give $\norm{v^*}\le1/\mu=\kappa$. Comparing the optimality equations at $x,x'$ also gives
$\norm{z^*(x)-z^*(x')}\le\kappa\norm{x-x'}$.

For $a\in\{x,y\}$, the fundamental theorem of calculus gives
\begin{equation}\label{eq:integral-field}
 \lambda[\nabla_a g_j(x,z+v/\lambda)-\nabla_a g_j(x,z)]
 =\int_0^1\nabla^2_{ay}g_j(x,z+s v/\lambda)v\,ds.
\end{equation}
Hessian Lipschitzness bounds the context Lipschitz constant by $O(1+\norm v)=O(\kappa)$ on $\norm v\le2\kappa$. Gradient Lipschitzness gives the global $O(1)$ bound in $v$. Strong monotonicity of $B$ then yields
\[
 \norm{v^*(x,z)-v^*(x',z')}\lesssim
 \kappa^2\bigl(\norm{x-x'}+\norm{z-z'}\bigr).
\]
Combining these bounds proves the composite Lipschitz estimate for $C$.

\subsection{Proof of Lemma~\ref{lem:penalty}}
Let $z=z^*(x)$ and $y=y^*_{\lambda}(x)$ be the unique minimizers in \eqref{eq:penalty}. The envelope formula gives
\[
 \nabla H_\lambda(x)=\nabla_x f(x,y)
             +\lambda[\nabla_x g(x,y)-\nabla_x g(x,z)].
\]
The optimality equations give $v=\lambda(y-z)=v^*(x,z)$, proving \eqref{eq:exact-field}. Lemma~\ref{lem:fields} bounds the Lipschitz constants of $x\mapsto v^*(x,z^*(x))$ and $\nabla H_\lambda$ by $O(\kappa^3)$.

To bound the bias, set $H=\nabla^2_{yy}g(x,z)$ and $K=\nabla^2_{xy}g(x,z)$. Taylor expansion of the lower gradient and Lipschitzness of the upper gradient imply
\[
 \norm{Hv+\nabla_y f(x,z)}
 \le\frac{\norm v}{\lambda}+\frac{\norm v^2}{2\lambda}
 \lesssim\frac{\kappa^2}{\lambda}.
\]
Therefore
\[
 \norm{v+H^{-1}\nabla_y f(x,z)}\lesssim\kappa^3/\lambda.
\]
The implicit function theorem and the chain rule give
\[
 \nabla\Phi(x)=\nabla_x f(x,z)-KH^{-1}\nabla_y f(x,z).
\]
Combining this identity with \eqref{eq:integral-field} for the $x$ block proves the bias bound in \eqref{eq:penalty-bounds}.

Evaluating the first minimum in \eqref{eq:penalty} at $z$ gives $H_\lambda\le\Phi$. Conversely, for every $y$,
\[
 f(x,y)+\lambda[g(x,y)-g(x,z)]
 \ge\Phi(x)-\norm{y-z}+\frac{\lambda\mu}{2}\norm{y-z}^2
 \ge\Phi(x)-\frac{1}{2\lambda\mu}.
\]
Minimizing over $y$ proves the value bounds and
$H_\lambda(x_0)-\inf H_\lambda\le\Delta+\kappa/(2\lambda)$.

\section{Local Curvature and Evaluation Martingales}\label{app:local}
\subsection{Local inverse and correction error}
Use the reference matrices of Section~\ref{sec:analysis} at the old inner point and new context: $s=x$ for $A$, and $s=(x,z)$ for $B$ after the independent $z$ correction. The $B$ reference is the lower Hessian, with a $1/\lambda$-Lipschitz gradient perturbation from $f$.

\begin{lemma}[Local inverse]\label{lem:localinverse}
Let $F$ be the gradient of a globally $\mu_0$-strongly convex differentiable potential, and let $H\succeq\mu_0 I$. Suppose throughout the Euclidean ball of radius $r$ about $u_0$ the map $F$ obeys
\[
 \ip{F(u)-F(v)}{u-v}\ge\frac12\norm{u-v}_H^2,
 \qquad \norm{F(u)-F(v)}_{H^{-1}}\le2\norm{u-v}_H.
\]
For twice differentiable potentials, Jacobians between $H/2$ and $2H$ suffice. If
\[
 \norm a_{H^{-1}}\le r\sqrt{\mu_0}/16,
\]
then the root of $F(u)-F(u_0)+a=0$ lies inside that ball and satisfies
\[
 \norm{u-u_0}_H\le2\norm a_{H^{-1}}.
\]
The same conclusion holds around another point strictly inside the ball when the indicated $H$-ellipsoid remains in the ball.
\end{lemma}
\begin{proof}
Minimize the shifted potential on the $H$-ellipsoid of radius $r\sqrt{\mu_0}/2$ about $u_0$, which lies in the Euclidean ball. On its boundary, with $d=u-u_0$,
\[
 \ip{F(u)-F(u_0)+a}{d}
 \ge\frac12\norm d_H^2-\norm a_{H^{-1}}\norm d_H>0.
\]
The minimizer is interior and hence is the unique global root by strong convexity. Applying the same inequality at this root gives the norm bound.
\end{proof}

For one correction, condition on the contexts $s_0,s_1$, starting point $u_0$, and correction $h$ before drawing its batch. Put $p_i=F_i(s_1,u_0)-F_i(s_0,u_0)$, with $\norm{p_i}\le D$. Let $\bar u$ solve $F(s_1,\bar u)-F(s_0,u_0)+h=0$, and set $\bar d=\bar u-u_0$, $Q=\norm{\bar d}_H^2$. Each $F_i(s_1,\cdot)$ is a $\mu_0$-strongly monotone gradient with Lipschitz constant at most two, where $\mu_0\ge1/(2\kappa)$. Take symmetric references $\mu_0I\preceq\mathcal H_i\preceq2I$, $H=\E_i\mathcal H_i$, such that for all $u,v$ in the local ball
\begin{equation}\label{eq:localoscillation}
 \norm{F_i(s_1,u)-F_i(s_1,v)-\mathcal H_i(u-v)}
 \le\frac{\mu_0}{32}\norm{u-v}.
\end{equation}
For $A$ and $B$, the component lower Hessians satisfy \eqref{eq:localoscillation} by Hessian Lipschitzness of $g_j$ and the $1/\lambda$ gradient perturbation from $f_i$.

\begin{lemma}[Curvature-weighted correction]\label{lem:localweighted}
Suppose the population corrected root lies in the inner quarter of the local ball and $Q+D^2$ is small enough for the ellipsoids below to remain in that ball. Conditional on the fixed inputs, outside an event of probability at most $C d\exp(-c\Lambda)$, the computed point $u_S$ satisfies
\begin{equation}\label{eq:pathE}
 E:=\norm{u_S-\bar u}_H^2
 \le C\frac{\kappa\Lambda}{b}(Q+D^2)+C\kappa^2\tau^2.
\end{equation}
Here the residual of the computed sampled solve is at most $\tau$, and $b\ge C\kappa\Lambda$.
\end{lemma}
\begin{proof}
The matrices $H^{-1/2}\mathcal H_iH^{-1/2}$ are positive semidefinite with mean $I$ and norm at most $2/\mu_0\le4\kappa$. Matrix Chernoff \citep{tropp} gives
\[
 \frac34H\preceq\mathcal H_S\preceq\frac54H
\]
except with probability $2d\exp(-cb/\kappa)$. By \eqref{eq:localoscillation}, the averaged remainder has dual norm at most $\norm{u-v}_H/32$, proving the empirical local inequalities of Lemma~\ref{lem:localinverse}.

At the population endpoint, set $\chi_i=F_i(s_1,\bar u)-F_i(s_1,u_0)$ and $\zeta_S=(p_S-\bar p)+(\chi_S-\bar\chi)$. Both endpoints are fixed before $S$. Co-coercivity and local comparison give $\E_i\norm{\chi_i}^2\le2\ip{\bar\chi}{\bar d}\le4Q$. Thus the centered summands $Z_i$, whitened by $H^{-1/2}$, satisfy
\[
 \E_i\norm{Z_i}^2\le C\kappa(D^2+Q),\qquad
 \norm{Z_i}\le C\kappa\sqrt{D^2+Q}.
\]
Here $\norm{\bar d}\le\sqrt{2\kappa Q}$ and $\norm{H^{-1/2}}\le\sqrt{2\kappa}$. Vector Bernstein from rectangular matrix Freedman \citep{freedman} gives
\[
 \norm{\zeta_S}_{H^{-1}}
 \le C\left(\sqrt{\frac{\kappa\Lambda}{b}}+
                 \frac{\kappa\Lambda}{b}\right)\sqrt{D^2+Q}.
\]
The input bounds verified below place this residual below the local inverse threshold at $\bar u$. Lemma~\ref{lem:localinverse} localizes the empirical root and bounds its $H$-error by twice the displayed quantity. The solve error is at most $2\kappa\tau$ in Euclidean norm, proving \eqref{eq:pathE}.
\end{proof}

After the correction, write $d_t=u_t-u_{t-1}$, $J_t=\norm{d_t}_{H_t}^2$, and $\eta_t=M_t-M_{t-1}$ for the independent evaluation error. The same conditional vector Bernstein calculation gives
\begin{equation}\label{eq:pathV}
 V_t:=\norm{\eta_t}_{H_t^{-1}}^2
 \le C\frac{\kappa\Lambda}{b}(J_t+D_t^2)
\end{equation}
with the same failure bound, conditional on the completed correction before $R_t$. The unwhitened component increments have second moment at most $C(J_t+D_t^2)$ and norm at most $C\sqrt{\kappa(J_t+D_t^2)}$.

\subsection{Evaluation martingales with a random energy budget}
\begin{lemma}[Predictable energy peeling]\label{lem:peeling}
Suppose an evaluation martingale is sampled in batches of size $b$. Before batch $t$, a nonnegative quantity $W_t$ is known. Each of its $b$ centered, scaled summands has norm at most $C\sqrt{\kappa W_t}/b$, and the batch conditional second-moment sum is at most $CW_t/b$. Suppose also $\sum_tW_t\le W_{\max}$ for a deterministically bounded stopped process. If $b\ge C\kappa\Lambda$, then, simultaneously over epoch prefixes, except with a probability bounded by the union of $O(\log(e+W_{\max}/w_0))$ matrix-Freedman tails,
\begin{equation}\label{eq:peeling}
 \norm{M_t}^2\le\frac{C\Lambda}{b}\sum_{s\le t}W_s+\nu^2,
 \qquad w_0=\frac{b\nu^2}{C'\Lambda}.
\end{equation}
The constant $C'$ is sufficiently large; enlarging the constant in the tail threshold if necessary gives each tail probability at most $d\exp(-c\Lambda)$.
\end{lemma}
\begin{proof}
At a deterministic level $w$, stop before any batch that would raise cumulative $W$ above $w$. This is predictable before its evaluation draws. The stopped increments are bounded by $C\sqrt{\kappa w}/b$, with variance sum at most $Cw/b$. Rectangular matrix Freedman, applied to the vector as a $d$-by-$1$ matrix, gives uniformly over prefixes
\[
 \norm{M_t}\le C\left(\sqrt{\frac{w\Lambda}{b}}
                +\frac{\sqrt{\kappa w}\,\Lambda}{b}\right)
 \le C\sqrt{\frac{w\Lambda}{b}}.
\]
Peel over $w=w_0,2w_0,4w_0,\ldots$ beyond $W_{\max}$. For each realized prefix, the first level above its energy gives the same stopped process up to that prefix. This level is at most twice the energy when the energy exceeds $w_0$; otherwise the bound is $\nu^2$ for sufficiently large $C'$. A union bound proves \eqref{eq:peeling}.
\end{proof}

For the two trackers take $W_t=D_t^2+J_t$. For $C$, take
\[
 W^C_t=\kappa^2\norm{\Delta(x,z)_t}^2+\norm{\Delta v_t}^2.
\]
The unscaled centered increments have norm $O(\sqrt{W^C_t})$ and variance $O(W^C_t)$. The independent draws in $R_t$, conditional on the correction, permit the individual-index filtration required by the lemma.

\subsection{Localization and radius control}
The two trackers use the following scales:
\begin{center}
\begin{tabular}{c|c|c|c}
tracker & input bound $D_{\max}$ & population energy bound $Q_{\max}$ & actual displacement\\
\hline
$A$ & $\theta$ & $a_A\kappa\theta^2$ & $C_A\kappa\theta$\\
$B$ & $C_D\kappa^2\theta$ & $a_B\kappa^5\theta^2$ & $C_B\kappa^3\theta$
\end{tabular}
\end{center}
All constants are absolute.

\begin{lemma}[Localization and radius control]\label{lem:induction}
With the parameter choices above, there is an event $\mathcal T$ of probability at least $1-\delta$ on which all inner residual tests and movement/radius checks pass throughout the $T$ steps. On this event, the table holds, consecutive reference metrics are comparable within a factor two, and \eqref{eq:pathE}, \eqref{eq:pathV}, and \eqref{eq:peeling} hold for all relevant prefixes. The constants $c_\theta,C_\lambda,c_\tau,c_\nu,C_b$ can be chosen independently of $N,\kappa,\varepsilon,d$.
\end{lemma}
\begin{proof}
Proceed in sampling order. For $A$, use radius $r_A=1/(128\kappa)$ about the old $z$, giving component Hessian oscillation at most $1/(128\kappa)$. For $B$, use radius $r_B=8\kappa$ about the old $v$. The lower and upper contributions to \eqref{eq:localoscillation} have Lipschitz bounds $8\kappa/\lambda$ and $1/\lambda$, whose sum is at most $\mu/128$ for large $C_\lambda$.

Assume the preceding bounds hold and $\norm{v_{t-1}}\le2\kappa$. Then
\begin{equation}\label{eq:nonlinearfeedback}
 F(s_t,\bar u_t)=M_{t-1},\qquad
 F(s_t,u_t)=M_{t-1}+k_t,\qquad
 \widehat F_t=k_t-\eta_t,
\end{equation}
where $e_t=u_t-\bar u_t$ includes the solve error and $k_t=F(s_t,u_t)-F(s_t,\bar u_t)$ satisfies $\norm{k_t}_{H_t^{-1}}\le2\norm{e_t}_{H_t}$. Hence the next population forcing is
\[
 a_{t+1}=\bar p_{t+1}+k_t-\eta_t.
\]
The preceding displacements bound consecutive reference differences by $\theta+C_A\kappa\theta$ for $A$. For $B$, the bound is
\[
 \theta+C_A\kappa\theta+C_B\kappa^3\theta/\lambda.
\]
Small $c_\theta$ and large $C_\lambda$ make these small multiples of $\mu$. Metric comparison and the bound on $k_t$ give
\begin{equation}\label{eq:localforcing}
 \norm{a_{t+1}}_{H_{t+1}^{-1}}^2
 \le C\big(\kappa D_{t+1}^2+E_t+V_t\big).
\end{equation}
At an epoch start, the initial residual replaces the last two terms by $O(\kappa\tau^2)$.

Put $\beta=C\kappa\Lambda/b$. If the preceding population energy is at most $Q_{\max}$, \eqref{eq:pathE} and \eqref{eq:pathV} give
\[
 E_t\le\beta(Q_{\max}+D_{\max}^2)+C\kappa^2\tau^2,
 \quad J_t\le2Q_{\max}+2E_t,
 \quad V_t\le\beta(J_t+D_{\max}^2).
\]
For $A$, take $D_{\max}=\theta$, choose $a_A$ larger than twice the forcing constant, and take $\beta$ small enough to absorb its coefficient. Since $\kappa^2\tau^2\ll\kappa\theta^2$, the dual forcing is $O(\sqrt\kappa\theta)$, below $r_A\sqrt{\mu/2}/16$ for small $c_\theta$. Lemma~\ref{lem:localinverse} gives $Q_t\le a_A\kappa\theta^2$ and localizes the population root. Lemma~\ref{lem:localweighted} then localizes the empirical root and gives \eqref{eq:pathE}; small $\beta,c_\tau$ keep its ellipsoid inside the ball. Thus $\norm{\Delta z_t}\le C_A\kappa\theta<1/(512\kappa)$.

Conditional on this $z$ correction, draw $S^v_t$ independently. The context bound gives $D^B_t\le C_D\kappa^2\theta$. The same argument with large $a_B$ bounds the dual forcing by $O(\kappa^{5/2}\theta)$, population energy by $a_B\kappa^5\theta^2$, and displacement by $C_B\kappa^3\theta$. Since $\kappa^{5/2}\theta=c_\theta\sqrt\kappa$ and $r_B\sqrt{\mu/2}\asymp\sqrt\kappa$, these lie inside the $B$ ball. The curvature comparisons and concentration ellipsoids hold before the next evaluation batch.

To prove the radius, apply Lemma~\ref{lem:peeling} to the evaluation process stopped before any correction that has failed the preceding induction. On every active prefix,
\[
 \sum_{s\le t}(D_{B,s}^2+J_{B,s})\le Cq\kappa^5\theta^2,
 \qquad
 \norm{M^B_t}^2\le C\frac{\Lambda q\kappa^5\theta^2}{b}+\nu^2
 \le \frac{C c_\theta^2}{C_b\Lambda}+\nu^2.
\]
Choose the constants to make $\norm{M^B_t}\le1/32$. Strong monotonicity and the first identity of \eqref{eq:nonlinearfeedback} imply
\[
 \norm{\bar v_t-v^*(x_t,z_t)}\le2\kappa\norm{M^B_{t-1}}.
\]
Furthermore, \eqref{eq:pathE} and the energy cap give
\[
 \norm{e^B_t}\le C\sqrt\beta\,\kappa^3\theta+C\kappa^{3/2}\tau\le\kappa/16.
\]
Together with $\norm{v^*}\le\kappa$, this gives $\norm{v_t}\le\kappa+\kappa/16+\kappa/16<2\kappa$, using only the preceding martingale value.

For the stopped process, suppress an evaluation before drawing $R_t$ whenever its correction violates a local bound. If an evaluation inequality fails, include that evaluation and stop before the next correction. The draws remain conditionally centered, and active inputs have deterministic energy ceilings $Cq\kappa\theta^2$, $Cq\kappa^5\theta^2$, and $Cq\kappa^6\theta^2$ for $A,B,C$. Thus Lemma~\ref{lem:peeling} applies without conditioning on future events.

There are $O(T)$ concentration events and epoch martingales. The polynomial energy ceilings and floor give $O(\log(e+W_{\max}/w_0))=O(\Lambda)$ peeling levels under \eqref{eq:parameters}--\eqref{eq:confidence}. Conditional union bounds, with a large constant in each tail threshold, give total probability at most $\delta/2$. Allocate $\delta/2$ to capped solver failures. On the joint complement the induction continues throughout, agrees with the algorithm, and all safeguards pass.
\end{proof}

\section{Proof of the Epoch Estimate}\label{app:epoch}
Index within-epoch transitions by their endpoints, $\Delta u_t=u_t-u_{t-1}$. The nonnegative final outer movement may be added to the right-hand side.
\begin{proof}
Use the energies in \eqref{eq:energy-definitions} and $V_t$ in \eqref{eq:pathV}. For one tracker the preceding lemmas give
\begin{align*}
 E_t&\le\beta(Q_t+D_t^2)+C\kappa^2\tau^2,&
 J_t&\le2Q_t+2E_t,\\
 V_t&\le\beta(J_t+D_t^2),&
 Q_t&\le C\kappa D_t^2+C E_{t-1}+C V_{t-1},
\end{align*}
with an initial additive $C\kappa\tau^2$ and $\beta=C\kappa\Lambda/b$. Summing and substituting the bounds for $E_t,V_t,J_t$ leaves a coefficient $C\beta<1/2$ on $\sum_t(Q_t+J_t)$. Absorbing it gives, for $U=\sum_tD_t^2$,
\begin{equation}\label{eq:nonlinearenergy}
 \sum_t(Q_t+J_t)\le C\kappa U+C\kappa^2h\tau^2,
 \quad
 \sum_t\norm{d_t}^2\le C\kappa^2U+C\kappa^3h\tau^2.
\end{equation}
The identity $F(s_t,\bar u_t)=M_{t-1}$ and strong monotonicity give
\[
 \norm{u_t-u^*(s_t)}^2\le C\kappa^2\norm{M_{t-1}}^2+C\kappa E_t.
\]
Insert \eqref{eq:peeling} with $W_t=J_t+D_t^2$, sum over $h\le q$ times, and use $b\ge C\kappa q\Lambda$:
\begin{equation}\label{eq:nonlineartracking}
 \sum_t\norm{u_t-u^*(s_t)}^2
 \le C\frac{\kappa^3q\Lambda}{b}U
       +C\kappa^3h\tau^2+C\kappa^2h\nu^2.
\end{equation}
For $A$, $U_A\le X:=\sum_t\norm{\Delta x_t}^2$. Hence
\[
 \sum_t\norm{\Delta z_t}^2\le C\kappa^2X+C\kappa^3h\tau^2.
\]
For $B$, $U_B\le C\kappa^2\sum_t\norm{\Delta(x,z)_t}^2
\le C\kappa^4X+C\kappa^5h\tau^2$. Equations \eqref{eq:nonlinearenergy} and \eqref{eq:nonlineartracking} imply
\begin{align*}
 \sum_t\norm{\Delta v_t}^2&\le C\kappa^6X+C\kappa^7h\tau^2,\\
 \sum_t\norm{v_t-v^*(x_t,z_t)}^2
 &\le C\frac{\kappa^7q\Lambda}{b}X+C\kappa^7h\tau^2+C\kappa^2h\nu^2.
\end{align*}
The term $\kappa^8qh\Lambda\tau^2/b$ arising in the substitution is at most $C\kappa^7h\tau^2$ by the batch choice.

For the $C$ evaluation martingale, peeling and the movement estimates give
\[
 \sum_t\norm{M^C_t}^2
 \le C\frac{\kappa^6q\Lambda}{b}X+C\kappa^6h\tau^2+Ch\nu^2.
\]
Finally, Lemma~\ref{lem:fields} gives
\[
 \norm{\widehat G_t-\nabla H_\lambda(x_t)}
 \lesssim\norm{M^C_t}+\norm{v_t-v^*(x_t,z_t)}
       +\kappa^2\norm{z_t-z^*(x_t)}.
\]
Squaring, summing, and applying the $A$ tracking bound with multiplier $\kappa^4$ proves \eqref{eq:epoch-main}. The first full estimate has error at most $C\kappa^3\tau$, included in the tolerance term. Passing to the epoch's input iterates discards only the unused final tracking estimate.
\end{proof}

\section{Accuracy of the Returned Point}\label{app:accuracy}
Let $e_t=\widehat G_t-\nabla H_\lambda(x_t)$ and let $\mathcal U$ be the output set. Smoothness with $L_H\alpha\le1/2$, $\ip{\widehat G_t}{g_t}\ge\norm{g_t}^2$, and Young's inequality give \eqref{eq:main-descent}.

\begin{proof}[Accuracy]
On $\mathcal T$, the tolerance term in \eqref{eq:epoch-main} is at most $c_0h\varepsilon^2$ for arbitrarily small $c_0$ by choosing $c_\tau,c_\nu$. The actual epoch lengths sum to $T$, also when $q>T$. Since $\Delta x=-\alpha g$,
\[
 C\alpha^2\kappa^7q\Lambda/b
 =C c_\alpha^2\kappa q\Lambda/b
\]
is sufficiently small. Summing the epoch estimates and \eqref{eq:main-descent}, then absorbing the movement term, gives pathwise
\begin{equation}\label{eq:pathdescent}
 \sum_{t<T}\norm{g_t}^2\le C\bar\Delta_H/\alpha+C c_0T\varepsilon^2,
 \qquad
 \sum_{t<T}\norm{e_t}^2\le C\bar\Delta_H/\alpha+C c_0T\varepsilon^2.
\end{equation}
Each clipped step contributes $\Gamma^2$ to the first sum. Since $\Gamma\ge10\kappa$ and $T\asymp C_T\kappa^3\varepsilon^{-2}$, large $C_T$ and small $c_0$ give at most $T/2$ clipped steps. Hence $|\mathcal U|\ge T/2$ pathwise on $\mathcal T$. Conditional on the trajectory,
\[
 \E_{\rm output}\norm{\nabla H_\lambda(x_{\rm out})}^2
 \le\frac4T\sum_{t<T}(\norm{g_t}^2+\norm{e_t}^2)
 \le C\frac{\bar\Delta_H}{\alpha T}+C c_0\varepsilon^2
 \le\varepsilon^2/8.
\]
On every outcome, including failure, the returned point is either $x_0$ or an outer iterate within $T\theta$ of $x_0$. Smoothness and lower boundedness imply
\[
 \norm{\nabla H_\lambda(x_0)}^2\le2L_H\bar\Delta_H,
 \qquad
 \norm{\nabla H_\lambda(x_{\rm out})}^2\le B_T^2.
\]
Thus $\mathcal T^c$ contributes at most $\delta B_T^2\le\varepsilon^2/64$. Choose $C_\lambda$ so that $\norm{\nabla H_\lambda-\nabla\Phi}\le\varepsilon/4$. The squared triangle inequality proves \eqref{eq:main-guarantee}.
\end{proof}

\section{Inner Solves and Oracle Complexity}\label{app:cost}
Each sampled problem has $b$ strongly convex, smooth components, so a capped accelerated finite-sum solve costs
\[
 \Ot(b+\sqrt{b\kappa})=\Ot(b)
\]
gradient calls, including residual verification. For failure probability $p\le\delta/(C(T+1))$, run to mean objective gap at most $p\tau^2/(2L)$. The bound $\norm{\nabla\psi}^2\le2L(\psi-\psi^*)$ and Markov's inequality give residual failure probability at most $p$, at logarithmic extra cost.

For full $B$ solves, write the potential as the sum of the averages of
\begin{align*}
 P_i(v)&=\lambda f_i(x,z+v/\lambda)+\frac\mu8\norm v^2,\\
 Q_j(v)&=\lambda^2[g_j(x,z+v/\lambda)-g_j(x,z)]
          -\lambda\ip{\nabla_y g_j(x,z)}v-\frac\mu8\norm v^2.
\end{align*}
For $\lambda\ge8\kappa$, $P_i,Q_j$ have strong convexity at least $\mu/8,3\mu/4$ and absolute smoothness constants: the nonquadratic part of $P_i$ is $1/\lambda$-weakly convex. The quadratics cancel between the averages. Let
\[
 r_f=\lceil N/n\rceil,\quad r_g=\lceil N/m\rceil,\quad
 M=nr_f+mr_g\in[2N,3N].
\]
Average $r_f$ copies of each $(M/(nr_f))P_i$ and $r_g$ copies of each $(M/(mr_g))Q_j$. The weights lie in $[1,3]$, giving $O(N)$ components of condition number $O(\kappa)$ and $O(1)$ gradient cost. A common strongly convex quadratic may be reserved for the solver. For $N>\kappa$, full epoch solves therefore cost $\Ot(N+\sqrt{N\kappa})=\Ot(N)$.

The caps are uniform over accepted histories. Between refreshes, the checks bound $z$'s movement by $O(q/\kappa)$ and $\norm v$ by $2\kappa$. Paired increments cancel component linear offsets. The initial gradient of a shifted sampled $B$ problem is $O(\kappa q+\tau)$; for $A$ it is polynomial in $q,\kappa$. Strong convexity bounds initial gaps by squared gradient norms divided by $\mu$. Scheduled warm starts obey analogous bounds since $\Lip(z^*)\le\kappa$ and $\norm{v^*}\le\kappa$. At initialization, a full gradient bounds the initial gap by its squared norm divided by twice the strong convexity parameter. Thus observed data fix the caps, with only logarithmic dependence on the initial residual or distance.

There are at most $T$ ordinary steps and $1+\lfloor T/q\rfloor$ full starts. Failed checks terminate the run, so the capped cost is
\[
 \Ot\left(N+T\left(\frac Nq+b\right)\right)
 =\Ot\left(N+\sqrt N\kappa^{7/2}\varepsilon^{-2}\right).
\]
This proves Theorem~\ref{thm:main} for $N>\kappa$.

\subsection{The full-refresh regime and Proposition~\ref{prop:baseline}}
Set $q=1$ in Algorithm~\ref{alg:main}; no sampled corrections or recursive evaluations are performed. On the event that all full solves meet their residual tolerance, strong monotonicity gives
\[
 \norm{z_t-z^*(x_t)}\le\kappa\tau,
 \qquad \norm{v_t-v^*(x_t,z_t)}\le2\kappa\tau.
\]
Lemma~\ref{lem:fields} therefore yields
$\norm{\widehat G_t-\nabla H_\lambda(x_t)}\lesssim\kappa^3\tau$.
With $\tau=c_\tau\varepsilon\kappa^{-4}$, the error is an arbitrarily small multiple of $\varepsilon$. Appendix~\ref{app:accuracy}, total solver failure probability $\delta$, and the same penalty bias give \eqref{eq:criterion} in $T=O(\kappa^3\varepsilon^{-2})$ steps.

Each step costs $\Ot(N+\sqrt{N\kappa})$, including full gradients and residual checks, proving Proposition~\ref{prop:baseline}. For $N\le\kappa$,
\[
 N\kappa^3\le\sqrt N\kappa^{7/2},
\]
which proves Theorem~\ref{thm:main} in the remaining regime.

\section{Lower Bounds}\label{app:lower}
\subsection{Proof of Proposition~\ref{prop:lower}}
Let $h=n^{-1}\sum_i h_i$ be an individually $1$-smooth finite sum with initial gap at most one. Set $m=1$, $n=N-1$, and
\[
 f_i(x,y)=h_i(x),\qquad
 g_1(x,y)=\frac12(y_1^2+\kappa^{-1}y_2^2).
\]
All assumptions hold with normalized condition number exactly $\kappa$ and $\Phi=h$. Each upper query requires one component query to $h$; lower replies are known. Theorem 3.1 of \citet{peng} gives $\Omega(n+\sqrt n\eta^{-2})$ calls for unrestricted randomized queries and arbitrary outputs at sufficiently small constant failure probability. Since $\E\norm{\nabla h}^2\le\varepsilon^2$ implies
$\Pr\{\norm{\nabla h}>a\varepsilon\}\le a^{-2}$ for fixed $a>0$, take $\eta=a\varepsilon$ with sufficiently large $a$. The expected-budget extension covers random stopping. Substituting $n=N-1$ proves the rate over worst-case decompositions of total size $N$, rather than every prescribed $(n,m)$.

\subsection{An unrestricted randomized joint lower bound}\label{app:unrestricted-joint}
We modify the dense weak hiding chain of \citet{peng} to bound its average gradient, then amplify stationarity through an affine lower solution.

\begin{proposition}[Joint lower bound]\label{prop:unrestricted-joint}
There are absolute constants $c,\varepsilon_0>0$ such that, for
$n\ge1024$, $m\ge1$, $\kappa\ge2$, and
$0<\varepsilon\le\varepsilon_0$, every randomized component-gradient
algorithm satisfying \eqref{eq:criterion} on all problems obeying
Assumptions~\ref{ass:smooth}--\ref{ass:gap}, with $\Delta=1$, has
worst-case complexity at least
\begin{equation}\label{eq:unrestricted-joint-lower}
 c\bigl(n+m+\sqrt n\kappa^2\varepsilon^{-2}\bigr).
\end{equation}
This holds for fixed and worst-case expected budgets, arbitrary adaptive component indices and query points, repeated queries, and arbitrary measurable outputs. Dimension may grow with the parameters and budget. For prescribed $N=n+m$, take $m=1$ to obtain $\Omega(N+\sqrt N\kappa^2\varepsilon^{-2})$ for $N\ge1025$.
\end{proposition}

\paragraph{A bounded-gradient hidden chain.}
Use the independent encoding tables and gate of \citet[Appendices A--C]{peng}. For $z=(z_1,\ldots,z_K)$, $z_t\in\R^D$, set
\[
 \Psi(z)=\frac{z}{\sqrt{1+\norm z^2}},\qquad
 q_t(z)=\ip{u_t}{\Psi(z)},\qquad
 \widetilde q_{it}(z)=\rho^{-1}\ip{s_{it}}{\Psi(z)}.
\]
The unit sign vectors $u_t,s_{it}$ satisfy $n^{-1}\sum_i s_{it}=\rho u_t$. Choose the least parity-compatible integer $m_\rho\ge\sqrt n$ and $\rho=m_\rho/n$, giving $n^{-1/2}\le\rho<17/(16\sqrt n)\le1/8$.
The nondecreasing gate $G$ obeys $G=G'=0$ below $5/32$, $G=1$ above $9/32$, $\norm{G'}_\infty\le17$, and $\norm{G''}_\infty\le259$. Replace the quadratic radial term by
\[
 H_5(z)=\begin{cases}\norm z^2/2,&\norm z\le5,\\
                         5\norm z-25/2,&\norm z>5,
          \end{cases}
 \qquad \nabla H_5(z)=z\min\{1,5/\norm z\},
\]
with zero gradient at $z=0$. For $\ell=33/16$, define
\begin{equation}\label{eq:bounded-hidden-chain}
\begin{aligned}
 \widetilde R_i(z)&=-\widetilde q_{i1}(z_1)
 -\sum_{t=2}^K G(q_{t-1}(z_{t-1}))
                  (\widetilde q_{it}(z_t)+1)\\
 &\quad+\ell\sum_{t=1}^K H_5(z_t),\qquad
 \widetilde R=\frac1n\sum_i\widetilde R_i.
\end{aligned}
\end{equation}
Then
\begin{equation}\label{eq:bounded-hidden-properties}
\begin{aligned}
 \Lip(\nabla\widetilde R_i)&\le H/\rho,
 &\widetilde R(0)-\inf\widetilde R&\le2K,\\
 \min_tq_t(z_t)<9/32&\ \Longrightarrow\ 
                    \norm{\nabla\widetilde R(z)}\ge\chi,
 &\sup_z\norm{\nabla\widetilde R(z)}&\le46\sqrt K,
\end{aligned}
\end{equation}
where $H=341$ and
$\chi=((1-(9/32)^2)^2-\ell(9/32))/\sqrt{1-(9/32)^2}>5/18$.
The $1$-Lipschitz Huber gradient preserves the $H/\rho$ bound by the block-row argument of their Lemma C.3. The gap bound follows from $-q_1-\sum_{t\ge2}G(q_{t-1})(q_t+1)\ge-2K$, $H_5\ge0$, and $\widetilde R(0)=0$.

The first unfinished block has gradient $\ell\nabla H_5(z_t)-\gamma_t\nabla q_t(z_t)$, where $1\le\gamma_t\le35$. Their Lemma A.6 gives the lower bound $\chi$ inside the radius-five ball; outside, $\norm{\nabla q_t(z_t)}\le(1+\norm{z_t}^2)^{-1/2}$ gives $5\ell-35/\sqrt{26}>3$. Each block gradient has norm at most $5\ell+35<46$, proving the final bound under the Lipschitz upper-gradient assumption in Assumption~\ref{ass:smooth}.

For $\eta>0$, set
\begin{equation}\label{eq:bounded-hidden-scaling}
 \beta=\frac{\rho\chi}{8H\eta},\qquad
 \alpha=\frac{64H\eta^2}{\rho\chi^2},\qquad
 K=\left\lfloor c_*\frac{\rho}{\eta^2}\right\rfloor,
 \qquad F_i(y)=\alpha\widetilde R_i(\beta y),
\end{equation}
where $0<c_*\le\chi^2/(256H)$ is an absolute constant. When $K\ge2$,
\eqref{eq:bounded-hidden-properties} gives
\[
 \Lip(\nabla F_i)\le1,\qquad F(0)-\inf F\le\frac12,\qquad
 \sup_y\norm{\nabla F(y)}
 \le\frac{46}{2\sqrt H}\sqrt\rho<1,
 \quad F=\frac1n\sum_iF_i.
\]
An unfinished block gives $\norm{\nabla F}\ge8\eta$. The table-independent replacement preserves exact reply locality and conditions (R1)--(R4) of their Appendix B. Their optimized-bias Lemma B.2, Theorem B.5, Corollary B.6, and randomized conversion in Appendix F.1 therefore give, for every randomized algorithm with at most $c_1\sqrt n\eta^{-2}$ component calls, a fixed instance satisfying
\begin{equation}\label{eq:bounded-hidden-failure}
 \Pr\{\norm{\nabla F(\widehat y)}\ge4\eta\}\ge11/16.
\end{equation}
The imported information argument charges $r=\lfloor n/512\rfloor$ calls per revealed stage over $K=\Theta(1/(\sqrt n\eta^2))$ stages, reserving the final stage for an output that need not have been queried. Their logarithmic choice of $D$ controls premature alignment at arbitrary query points; the argument thus applies to this modified chain, without asserting a theorem for arbitrary bounded-gradient functions.

\begin{proof}[Proof of Proposition~\ref{prop:unrestricted-joint}]
Take $\mu=\kappa^{-1}$, $s=\kappa-1$, and the construction above with $\eta=\varepsilon/s$. Set
\begin{equation}\label{eq:affine-lower-amplification}
 f_i(x,y)=F_i(y),\qquad
 g_j(x,y)=\frac\mu2\norm y^2-(1-\mu)\ip xy
 \quad(j=1,\ldots,m).
\end{equation}
Each $f_i$ is jointly $1$-smooth, $\norm{\nabla_yf}\le1$, and the lower Hessian is
\[
 \begin{pmatrix}0&-(1-\mu)I\\-(1-\mu)I&\mu I\end{pmatrix}.
\]
Its eigenvalues $(\mu\pm\sqrt{\mu^2+4(1-\mu)^2})/2$ have modulus at most one, and its $yy$ block is $\mu I$. Lower smoothness, componentwise strong convexity, and Hessian Lipschitzness follow. Also $y^*(x)=sx$, $\Phi(x)=F(sx)$, and the initial gap is at most $1/2$.

Simulate each upper query at $(x,y)$ by one query to $F_i$ at $y$; the known quadratic answers lower queries without accessing $F$. Preserve private randomness and return $\widehat y=sx_{\rm out}$. Since $\nabla\Phi(x)=s\nabla F(sx)$, \eqref{eq:bounded-hidden-failure} gives failure probability at least $11/16$ at threshold $4\varepsilon$ for budgets at most $c_1\sqrt n\,s^2\varepsilon^{-2}$. The iterative term follows from $s\ge\kappa/2$. If $K<2$, scaling gives $\sqrt n\kappa^2\varepsilon^{-2}=O(n)$, covered by the linear bound below.

For the $n$ term, draw independent uniform $i^*\in[n]$ and $\sigma\in\{-1,1\}$, and set
\[
 f_i(x,y)=\frac\gamma2x^2-n\tau\sigma\mathbf1_{\{i=i^*\}}x,
 \qquad g_j(x,y)=\frac\mu2y^2,\qquad
 \tau=2\varepsilon,\quad\gamma=2\varepsilon^2.
\]
All assumptions hold with gap $\tau^2/(2\gamma)=1$. Until $i^*$ is queried, the sign is unknown and one sign gives $|\gamma x-\tau\sigma|\ge2\varepsilon$ at any output. With at most $n/4$ calls, the index is missed with probability at least $3/4$, giving gradient failure probability at least $3/8$.

For the $m$ term, draw independent uniform $j^*\in[m]$ and $\sigma$, and set
\[
 f_i(x,y)=\frac12 H_1(x-y),\qquad
 g_j(x,y)=\frac\mu2y^2-m\mu\sigma\mathbf1_{\{j=j^*\}}y,
\]
where $H_1$ is the scalar radius-one Huber function. The joint upper gradient is one-Lipschitz, $|\nabla_yf|\le1/2$, $y^*=\sigma$, and the gap is $1/4$.
Since $\nabla\Phi(x)=\frac12\operatorname{clip}(x-\sigma,[-1,1])$, one sign has gradient magnitude $1/2$ at any output. At most $m/4$ calls give failure probability at least $3/8$. Take $\varepsilon_0\le1/8$ so that both needles contradict \eqref{eq:criterion}, and use that the maximum of three bounds is at least one third of their sum.

For worst-case expected budget $B$, cap the algorithm at $32B$ calls. Markov's inequality changes its output with probability at most $1/32$. Subtracting this from the fixed-budget bounds gives expected squared gradients at least $16(11/16-1/32)\varepsilon^2$ for the amplified family and $4(3/8-1/32)\varepsilon^2$ for the upper needle. The lower needle has the same probability bound at threshold $1/2$. Each contradicts \eqref{eq:criterion}; reducing the absolute constant proves the expected-budget claim.
\end{proof}

The square-root term uses the upper sum size $n$; for $m\gg n$ the construction gives no $\sqrt m$ term. When $n\asymp N$, a factor $\kappa^{3/2}$ remains relative to SICBO's upper bound. The conjectured joint $\kappa^{5/2}$ lower bound for unrestricted randomized algorithms remains open.

\section{Details of the Numerical Experiments}\label{app:numerics}
\subsection{Synthetic problem}\label{app:num-synthetic}
In \eqref{eq:comparison-family}, $w$ is a random unit vector. Each $H_j$ has eigenvalue $2\mu+I_j/4$ along $w$, with $\Pr(I_j=1)=4\mu$, and independently rotated eigenvalues in $[0.1,0.5]$ on the complement of $w$. Thus the component Hessians do not commute. $M$ has random singular vectors and geometrically spaced singular values from $0.6$ to $0.06$; $\norm{E_j}=0.08$; $b_j$, $t_i$ are Gaussian. Each $g_j(x,\cdot)$ is $\mu$-strongly convex with joint Hessian norm at most $0.94$ and $|\psi'''|\le\sqrt\mu$; each $f_i$ is $0.7$-smooth with $\norm{\nabla_yf}<1$.

Katyusha uses \citet{katyusha}'s parameters and snapshots for residual checks; the exact $\nabla\Phi$ used for plotting is uncounted. Relative to \eqref{eq:parameters}, we use $\alpha=1/\hat L$ with a sampled estimate $\hat L$ of $\Lip(\nabla\Phi)$, replace $C_b\Lambda^2$ by $1$, and monitor safeguards without enforcing them. The $z$ displacement bound was exceeded on $55$--$90\%$ of sampled steps; the $v$ bound always held and clipping never activated. Full-refresh and SICBO costs average ten seeds. The ratio of full-refresh to SICBO cost per outer iteration decreases with $\kappa$ approximately in proportion to $\sqrt{N/\kappa}$.

\subsection{Real data}\label{app:num-real}
For \eqref{eq:real-reg}, a9a and w8a use $5{,}000$ validation rows and the remainder for training; ijcnn1 uses its training and test files \citep{saba,srba}. At the tested $\kappa$, $\nabla^2_{yy}g_j\preceq(1/4+10\mu)I\preceq I$ and $\norm{\nabla_yf}\le1$. Strong convexity and $\norm{\nabla_y g(x,0)}\le1/2$ give $\norm{y^*(x)}\le\kappa/2$. The joint Hessian of $g_j$ contains $r_k'(x)y_k$ and $\tfrac12r_k''(x)y_k^2$, which are unbounded over $y$. Hence Assumption~\ref{ass:smooth} fails globally, and the experiments lie outside the normalized theoretical class. The choice $r_k(x)=e^{x_k}$ in \citet{saba,srba} does not permit prescribed $\kappa$.

We use $\tau=\max\{\varepsilon\kappa^{-4},10^{-11}\}$ and $\alpha=c/\hat L$, choosing $c\in\{1/4,1/2,1,2,4\}$ for each method and setting by the median cost over three seeds (one for deterministic methods). The selected value was $c=4$ in every setting: the bounded regularizers saturate and $\norm{\nabla\Phi(x_t)}$ decays roughly as $1/(\alpha t)$. The penalty methods stop after three times SICBO's mean cost at $c=1$; the penalty method never reached the target, and its accelerated variant reached it only on a9a and w8a. The accelerated penalty method and full-refresh benchmark approximate the same population outer update, with comparable per-iteration costs at $\kappa=100$ and differing by roughly $\sqrt{10}$ at $\kappa=1000$ (Table~\ref{tab:complexity}). Movement and radius safeguards were monitored without enforcement. The $z$ displacement bound was exceeded on 78--100\% of SICBO's sampled steps; the $v$ bound always held and clipping never activated.

For data hyper-cleaning on a9a, we flip $30\%$ of training labels and use $g_j(x,y)=\phi(b_j\ip{a_j}{y})/(1+e^{-x_j})+\frac\mu2\norm y^2$ with $x\in\R^m$ and $f_i$ from \eqref{eq:real-reg}. SICBO needed the fewest calls at both values of $\kappa$. Each $C_{ij}$ is supported on coordinate $j$; a sampled correction changes at most $b$ of the $m$ coordinates of $\widehat G_t$, yielding substantial departures from exact gradient descent. These objectives also fail global joint smoothness in Assumption~\ref{ass:smooth}.

\end{document}